\documentclass[a4paper,reqno,oneside,12pt]{amsart}
\usepackage[utf8]{inputenc}
\usepackage{fullpage}
\usepackage{amsfonts, amsmath, amssymb, amsthm}
\usepackage{enumitem}
\usepackage{graphicx}
\usepackage{textcomp}
\usepackage{etoolbox}
\usepackage{eso-pic}
\usepackage{verbatim}
\usepackage{comment}
\usepackage{tikz}
\usetikzlibrary{backgrounds, graphs, graphs.standard, quotes, 3d, patterns, calc, fit, shapes.geometric, decorations.pathreplacing, fadings, patterns, math}
\usepackage{mathtools}
\usepackage{thm-restate}
\usepackage{transparent}
\newcommand{\N}[1]{\mathcal{N}\left(#1\right)}

\newcommand{\eps}{\varepsilon}

\newcommand{\vass}[1]{\left| #1\right|}

\DeclareMathOperator{\exO}{ex}
\newcommand{\ex}[1]{\exO\left(#1\right)}

\renewcommand{\epsilon}{\varepsilon}

\theoremstyle{plain}
\newtheorem{theorem}{Theorem}

\theoremstyle{remark}

\theoremstyle{definition}

\usepackage[colorlinks]{hyperref}

\title{Density of small diameter subgraphs in $K_r$-free graphs}
\author[E. K. Hng]{Eng Keat Hng}
\address{(EKH) Czech Academy of Sciences, Institute of Computer Science, Pod Vod\'arenskou v\v{e}\v{z}\'i~2, 182~07 Prague, Czech~Republic}
\email{hng@cs.cas.cz}
\author[D. Mergoni Cecchelli]{Domenico Mergoni Cecchelli}
\address{(DMC) London School of Economics, Department of Mathematics, Houghton Street, London WC2A 2AE, UK}
\email{d.mergoni@lse.ac.uk}
\date{July 2022}
\thanks{Research supported by Czech Science Foundation project GX21-21762X and with institutional support RVO:67985807}

\begin{document}

\maketitle

\begin{abstract}
We denote by $\ex{n, H, F}$ the maximum number of copies of $H$ in an $n$-vertex graph that does not contain $F$ as a subgraph. 
Recently, Grzesik, Gy\H{o}ri, Salia, Tompkins considered conditions on $H$ under which $\ex{n, H, K_r}$ is asymptotically attained at a blow-up of $K_{r-1}$, and proposed a conjecture. In this note we disprove their conjecture.
\end{abstract}

\section{Introduction}

Let $G$, $H$ and $F$ be graphs. We say that $G$ is \emph{$F$-free} if it does not contain $F$ as a (not necessarily induced) subgraph. Write $\N{H,G}$ for the number of (unlabelled) copies of $H$ in $G$. We denote by $\ex{n,H,F}$ the generalized Tur\'an number, which is the maximum value of $\N{H,G}$ in an $F$-free graph $G$ on $n$ vertices. 

A prominent topic of research in extremal graph theory is the behaviour of $\ex{n,H,F}$. The classical case $H=K_2$ was studied by Tur\'an~\cite{Turan} for $F=K_r$ and then by Erd\H{o}s, Simonovits and Stone~\cite{ErdosSimonovits, ErdosStone} for general $F$. Alon and Shikhelman~\cite{AlonShikhelman} introduced and initiated the systematic study of $\ex{n,H,F}$ for general choices of $H$ and $F$.

Extending this line of study, Lidick\'y and Murphy~\cite{LidickyMurphy} conjectured that given a graph $H$ and an integer $r>\chi(H)$ (where $\chi(H)$ denotes the chromatic number of $H$),  there exists a complete $(r-1)$-partite graph with asymptotically as many copies of $H$ as possible in a $K_r$-free graph (which is by definition $\ex{n,H,K_r}$). Grzesik, Gy\H{o}ri, Salia and Tompkins~\cite{GrzesikGyoriSaliaTompkins} recently showed that for every $r\ge3$ there is a counterexample to the conjecture of Lidick\'y and Murphy. We need some extra definitions to describe their counterexamples.

Let us denote with $P_n$ the path on $n$ vertices and, given a graph $G$ and a positive integer $r$, let us write $G^r$ for the graph obtained from $G$ by joining every pair of vertices at distance at most $k$ in $G$. Moreover, given a graph $H$ and a vertex $v$ in $V(H)$, let the graph $H'$ be obtained from $H$ by replacing $v$ with an independent set $I_v$ of size $a$ and by adding a complete bipartite graph between $N_H(v)$ and $I_v$. We say that $H'$ is obtained from $H$ by \emph{blowing up} $v$ by a factor of $a$;  we also say that $H'$ is a blow-up of $H$.

As mentioned above, Grzesik, Gy\H{o}ri, Salia and Tompkins~\cite{GrzesikGyoriSaliaTompkins} constructed a counterexample to Lidick\'y and Murphy's conjecture. This was done by building a sequence $(F_r)_{r\ge3}$ of graphs where $F_r$ is obtained from $P^{r-2}_{2r}$ by blowing up its two endvertices by a large factor dependent on $r$. Let us note that Gerbner~\cite{Gerbner} very recently extended the construction of Grzesik, Gy\H{o}ri, Salia and Tompkins~\cite{GrzesikGyoriSaliaTompkins} to general $F$-free graphs. In the same paper, Grzesik, Gy\H{o}ri, Salia and Tompkins~\cite{GrzesikGyoriSaliaTompkins} proposed a new version of the conjecture.

\begin{restatable}[Grzesik, Gy\H ori, Salia and Tompkins~\cite{GrzesikGyoriSaliaTompkins}]{conjecture}{ConJ}\label{conjecture}
If $G$ is a graph with $\chi(G)<r$ and diameter at most $2r-2$, then $\ex{n, G, K_r}$ is asymptotically achieved by a blow-up of $K_{r-1}$.
\end{restatable}

\section{Counterexamples to conjecture \ref{conjecture}}

Our main result states that for $r\ge4$ there exists a counterexample to Conjecture~\ref{conjecture}.

\begin{theorem}\label{thm:ConjIsWrong}
For any $r\ge4$ and $\delta>0$ there is a graph $G$ of diameter $2$ such that for all sufficiently large $n$ and any $(r-1)$-partite graph $T$ on $n$ vertices we have
\[\delta\cdot\ex{n, G, K_{r}}>\N{G, T}.\]
\end{theorem}

\begin{proof}
Let $r\ge4$. Let $H$ be the graph obtained from the disjoint union of a copy $K$ of $K_{r-3}$ and a copy $P$ of $P_6$ by inserting all edges between $V(K)$ and $V(P)$. Write $x$ and $y$ for the two endvertices of $P$ in $H$. Several examples of $H$ are depicted in Figure~\ref{fig:counterexamples} with the vertices $x$ and $y$ marked. Fix a positive constant $\eps<\frac{1}{100(r+1)}$ and a positive integer $a$ such that $\frac{1}{2}\delta\eps^{r+1}(1-\eps(r+1))^{2a} \ge \frac{1}{2^{2a}}$. Write $G$ for the graph obtained from $H$ by blowing up both $x$ and $y$ by a factor of $a$. Let $X$ and $Y$ be the independent sets corresponding to $x$ and $y$ respectively. Observe that the diameter of $G$, which is equal to the diameter of $H$, is exactly $2$. Observe that $\chi(G)<r$.

Let $T$ be a complete $(r-1)$-partite graph on $n$ vertices. Let us count the number of labelled copies of $G$ in $T$. Observe that $G$ has a unique $(r-1)$-colouring (up to relabelling of the colours), and this colouring has the property that all vertices in $X$ get the same colour, all vertices in $Y$ get the same colour, and these two colours are different. Hence, there are at most $n^{r+1}\left(\frac{n}{2}\right)^{2a}$ labelled copies of $G$ in $T$.

On the other hand, let us denote by $Q$ the graph obtained from the disjoint union of a copy $K$ of $K_{r-3}$ and a copy $C$ of $C_5$ by inserting all edges between $V(K)$ and $V(C)$; let us note that $Q$ can be obtained from $H$ by contracting the vertices $x$ and $y$ to a new vertex $z$.  Write $S$ for the graph obtained from $Q$ by blowing up the vertex $z$ by a factor of $n-(r+1)\lfloor \epsilon n\rfloor$ and every other vertex by a factor of $\lfloor \epsilon n\rfloor$. Several examples of $S$ are depicted in Figure~\ref{fig:counterexamplesquotient}. Since $C_5$ is triangle-free, we have that $S$ is $K_r$-free. Now since we have at least $\lfloor \epsilon n\rfloor^{r+1}$ choices for the vertices of $Q\setminus z$ and at least $n-(r+1)\lfloor \epsilon n\rfloor$ choices for the vertices in $X \cup Y$, the number of labelled copies of $G$ in $S$ is at least
\[ \lfloor \epsilon n\rfloor^{r+1}\left( n-(r+1)\lfloor \epsilon n\rfloor\right)^{2a} - o(n^{\vass{G}}). \]
Then, by the definition of $\ex{n, G, K_r}$ and our choice of $\eps$ and $a$ we have
\[\delta \cdot \ex{n, G, K_r} \ge \delta \cdot \N{G,S} > \N{G,T}\]
as required.
\end{proof}

\begin{figure}
    \centering
\begin{tikzpicture}[scale=1]
\begin{scope}[shift={(0,0)}]
\foreach \x in {1, 2, 3, 4, 5, 6, 7, 8}{
\node[circle, draw, color=black, fill=black!20, inner sep=0pt, minimum width=4pt] (\x) at ($(0,0)+(\x*45:1cm)$) {};
}

\foreach \x/\y in {a/0, b/1, c/2, d/3, e/4, f/5}{
\node[circle, draw, color=black, fill=black!20, inner sep=0pt, minimum width=4pt] (\x) at ($(0,0)+(30+\y*60:2cm)$) {};}

\foreach \x in {2, ..., 8}{
    \foreach \y in {1, ..., \x}{
        \draw[] (\x) -- (\y);
    }
}    

\foreach \x in {a, b, c, d, e, f}{
    \foreach \y in {1, 2, 3, 4, 5, 6, 7, 8}{
        \draw[] (\x) -- (\y);
    }
}

\draw[line width=1.5pt] (a) -- (b) -- (c) -- (d) -- (e) -- (f);
\foreach \x in {1, 2, 3, 4, 5, 6, 7, 8}{
\node[circle, draw, color=black, fill=black!20, inner sep=0pt, minimum width=4pt] (\x) at ($(0,0)+(\x*45:1cm)$) {};
}
\node[circle, draw, fill=black!20, inner sep=0pt, minimum width=4pt, label={0:\tiny{$x$}}] () at ($(0,0)+(30:2cm)$) {};
\node[circle, draw, fill=black!20, inner sep=0pt, minimum width=4pt, label={0:\tiny{$y$}}] () at ($(0,0)+(-30:2cm)$) {};
\foreach \x in {1, 2, 3, 4, 5, 6, 7, 8}{
\node[circle, draw, color=black, fill=black!20, inner sep=0pt, minimum width=4pt] (\x) at ($(0,0)+(\x*45:1cm)$) {};
}
\end{scope}
\begin{scope}[shift={(-5,0)}]
\foreach \x in {1, 2, 3, 4}{
\node[circle, draw, color=black, fill=black!20, inner sep=0pt, minimum width=4pt] (\x) at ($(0,0)+(\x*90:1cm)$) {};
}

\foreach \x/\y in {a/0, b/1, c/2, d/3, e/4, f/5}{
\node[circle, draw, color=black, fill=black!20, inner sep=0pt, minimum width=4pt] (\x) at ($(0,0)+(30+\y*60:2cm)$) {};}

\foreach \x in {1, 2, 3, 4}{
    \foreach \y in {1, 2, 3}{
        \draw[] (\x) -- (\y);
    }
}    

\foreach \x in {a, b, c, d, e, f}{
    \foreach \y in {1, 2, 3, 4}{
        \draw[] (\x) -- (\y);
    }
}

\draw[line width=1.5pt] (a) -- (b) -- (c) -- (d) -- (e) -- (f);
\foreach \x in {1, 2, 3, 4}{
\node[circle, draw, color=black, fill=black!20, inner sep=0pt, minimum width=4pt] (\x) at ($(0,0)+(\x*90:1cm)$) {};
}
\node[circle, draw, fill=black!20, inner sep=0pt, minimum width=4pt, label={0:\tiny{$x$}}] () at ($(0,0)+(30:2cm)$) {};
\node[circle, draw, fill=black!20, inner sep=0pt, minimum width=4pt, label={0:\tiny{$y$}}] () at ($(0,0)+(-30:2cm)$) {};
\end{scope}
\begin{scope}[shift={(-10,0)}]
\foreach \x in {1}{
\node[circle, draw, color=black, fill=black!20, inner sep=0pt, minimum width=4pt] (\x) at ($(0,0)$) {};
}

\foreach \x/\y in {a/0, b/1, c/2, d/3, e/4, f/5}{
\node[circle, draw, color=black, fill=black!20, inner sep=0pt, minimum width=4pt] (\x) at ($(0,0)+(30+\y*60:2cm)$) {};}

\foreach \x in {a, b, c, d, e, f}{
    \foreach \y in {1}{
        \draw[] (\x) -- (\y);
    }
}

\draw[line width=1.5pt] (a) -- (b) -- (c) -- (d) -- (e) -- (f);

\node[circle, draw, fill=black!20, inner sep=0pt, minimum width=4pt, label={0:\tiny{$x$}}] () at ($(0,0)+(30:2cm)$) {};
\node[circle, draw, fill=black!20, inner sep=0pt, minimum width=4pt, label={0:\tiny{$y$}}] () at ($(0,0)+(-30:2cm)$) {};
\end{scope}
\end{tikzpicture}
    \caption{Graphs $H$ for $r=4, 7, 11$}
    \label{fig:counterexamples}
\end{figure}

\begin{figure}
    \centering
\begin{tikzpicture}[scale=1]
\begin{scope}[shift={(0,0)}]
\foreach \x in {1, 2, 3, 4, 5, 6, 7, 8}{
\node[circle, draw, color=black, inner sep=0pt, minimum width=7pt] (\x) at ($(0,0)+(\x*45:1cm)$) {};
}

\foreach \x/\y in {a/0, b/1, c/2, d/3, e/4}{
\node[circle, draw, color=black, inner sep=0pt, minimum width=7pt] (\x) at ($(0,0)+(36.5+\y*72:2cm)$) {};}

\foreach \x in {2, ..., 8}{
    \foreach \y in {1, ..., \x}{
        \draw[] (\x) -- (\y);
    }
}    

\foreach \x in {a, b, c, d, e}{
    \foreach \y in {1, 2, 3, 4, 5, 6, 7, 8}{
        \draw[] (\x) -- (\y);
    }
}

\draw[line width=1.5pt] (a) -- (b) -- (c) -- (d) -- (e) -- (a);
\foreach \x in {1, 2, 3, 4, 5, 6, 7, 8}{
\node[circle, draw, color=black, inner sep=0pt, minimum width=7pt] (\x) at ($(0,0)+(\x*45:1cm)$) {};
}
\node[circle, draw, fill=white, inner sep=0pt, minimum width=15pt] () at ($(0,0)+(36.5:2cm)$) {};
\foreach \x in {1, 2, 3, 4, 5, 6, 7, 8}{
\node[circle, draw, color=black, fill=white, inner sep=0pt, minimum width=7pt] (\x) at ($(0,0)+(\x*45:1cm)$) {};
}
\end{scope}
\begin{scope}[shift={(-5,0)}]
\foreach \x in {1, 2, 3, 4}{
\node[circle, draw, color=black, inner sep=0pt, minimum width=7pt] (\x) at ($(0,0)+(\x*90:1cm)$) {};
}

\foreach \x/\y in {a/0, b/1, c/2, d/3, e/4}{
\node[circle, draw, color=black, fill=white, inner sep=0pt, minimum width=7pt] (\x) at ($(0,0)+(36.5+\y*72:2cm)$) {};}

\foreach \x in {1, 2, 3, 4}{
    \foreach \y in {1, 2, 3}{
        \draw[] (\x) -- (\y);
    }
}    

\foreach \x in {a, b, c, d, e}{
    \foreach \y in {1, 2, 3, 4}{
        \draw[] (\x) -- (\y);
    }
}

\draw[line width=1.5pt] (a) -- (b) -- (c) -- (d) -- (e) -- (a);
\foreach \x in {1, 2, 3, 4}{
\node[circle, draw, color=black, inner sep=0pt, minimum width=7pt] (\x) at ($(0,0)+(\x*90:1cm)$) {};
}
\node[circle, draw, fill=white, inner sep=0pt, minimum width=15pt] () at ($(0,0)+(36.5:2cm)$) {};
\foreach \x in {1, 2, 3, 4}{
\node[circle, draw, color=black, fill=white, inner sep=0pt, minimum width=7pt] (\x) at ($(0,0)+(\x*90:1cm)$) {};
}
\end{scope}
\begin{scope}[shift={(-10,0)}]
\foreach \x in {1}{
\node[circle, draw, color=black, inner sep=0pt, minimum width=7pt] (\x) at ($(0,0)$) {};
}

\foreach \x/\y in {a/0, b/1, c/2, d/3, e/4}{
\node[circle, draw, color=black, inner sep=0pt, minimum width=7pt] (\x) at ($(0,0)+(36.5+\y*72:2cm)$) {};}

\foreach \x in {a, b, c, d, e}{
    \foreach \y in {1}{
        \draw[] (\x) -- (\y);
    }
}

\draw[line width=1.5pt] (a) -- (b) -- (c) -- (d) -- (e) -- (a);

\node[circle, draw, fill=white, inner sep=0pt, minimum width=15pt] () at ($(0,0)+(36.5:2cm)$) {};

\end{scope}
\end{tikzpicture}
    \caption{Graphs $S$ for $r=4, 7, 11$}
    \label{fig:counterexamplesquotient}
\end{figure}

\bibliography{bibliography}

\providecommand{\bysame}{\leavevmode\hbox to3em{\hrulefill}\thinspace}
\providecommand{\MR}{\relax\ifhmode\unskip\space\fi MR }
\providecommand{\MRhref}[2]{%
  \href{http://www.ams.org/mathscinet-getitem?mr=#1}{#2}
}
\providecommand{\href}[2]{#2}
\begin{thebibliography}{1}

\bibitem{AlonShikhelman}
N.~Alon and C.~Shikhelman, \emph{Many {$T$} copies in {$H$}-free graphs}, J.
  Combin. Theory Ser. B \textbf{121} (2016), 146--172.

\bibitem{ErdosStone}
P.~Erd\H{o}s and A.~H. Stone, \emph{On the structure of linear graphs}, Bull.
  Amer. Math. Soc. \textbf{52} (1946), 1087--1091.

\bibitem{ErdosSimonovits}
P.~Erd{\H{o}}s and M.~Simonovits, \emph{A limit theorem in graph theory},
  Studia Sci. Math. Hung, 1965.

\bibitem{Gerbner}
D.~Gerbner, \emph{On weakly {T}ur\'an-good graphs}, 2022.

\bibitem{GrzesikGyoriSaliaTompkins}
A.~Grzesik, E.~Gy{\H{o}}ri, N.~Salia, and C.~Tompkins, \emph{Subgraph densities
  in {$K_r$}-free graphs}, arXiv preprint arXiv:2205.13455 (2022).

\bibitem{LidickyMurphy}
B.~Lidick\'{y} and K.~Murphy, \emph{Maximizing five-cycles in {$K_r$}-free
  graphs}, European J. Combin. \textbf{97} (2021), Paper No. 103367, 29.

\bibitem{Turan}
P.~Tur\'an, \emph{Eine extremalaufgabe aus der graphentheorie}, Mat. Fiz. Lapok
  \textbf{48} (1941), 436--452.

\end{thebibliography}
\bibliographystyle{amsplain}

\end{document}